\documentclass[smallextended,envcountsame]{svjour3}
\smartqed
\usepackage{amsmath,amssymb}
\usepackage{graphicx}
\usepackage{psfrag}
\usepackage{txfonts}
\usepackage{xcolor}
\usepackage[all]{xy}
\usepackage{tikz}
\usetikzlibrary{intersections,calc}

\journalname{Japan J.\ Indust.\ Appl.\ Math.}

\numberwithin{equation}{section}

\newtheorem{miho}{Lemma~3${}^{new}$}

\newcommand{\R}{\mathbb R}
\newcommand{\bfx}{\mathbf x}

\newcommand{\PP}{\mathcal{P}}
\newcommand{\T}{\mathcal{T}}
\newcommand{\I}{\mathcal{I}}

\newcommand{\E}{\mathcal{F}}
\newcommand{\DG}{\mathcal{DG}}
\newcommand{\dd}{\mathrm{d}}
\newcommand{\tT}{\widetilde{T}}

\newcommand{\dsp}{\displaystyle}

\newcommand{\bnormal}{\mathbf n}
\newcommand{\asip}{a_h^{std}}
\newcommand{\asipnew}{a_{h}^{new}}
\newcommand{\jumpterm}{\left\{\frac{|f|}{|\widetilde{T}_f|}\right\}}
\newcommand{\revise}[1]{{\color{black} #1}}

\begin{document}

\title{A robust discontinuous Galerkin scheme on anisotropic meshes}

\author{Takahito Kashiwabara \and Takuya Tsuchiya}

\institute{Takahito Kashiwabara \at
           Graduate School of Mathematical Sciences, \\
           The University of Tokyo, Tokyo, Japan \\
           \email{\texttt{tkashiwa@ms.u-tokyo.ac.jp}} \and
           Takuya Tsuchiya \at
           Graduate School of Science and Engineering, \\
           Ehime University, Matsuyama, Japan \\
           \email{\texttt{tsuchiya@math.sci.ehime-u.ac.jp}}
}

\date{Received: date / Accepted: date}

\maketitle

\begin{abstract}
Discontinuous Galerkin (DG) methods are extensions of the usual
Galerkin finite element methods.
Although there are vast amount of studies on DG methods, most of them
have assumed shape-regularity conditions on meshes for both theoretical
error analysis and practical computations.
In this paper, we present a new
symmetric interior penalty DG scheme with a 
modified penalty term.  We show that, without imposing the
shape-regularity condition on the meshes, the new DG scheme inherits all
of the good properties of standard DG methods, and is thus
robust on anisotropic meshes.  Numerical experiments confirm
the theoretical error estimates obtained.
\end{abstract}

\noindent

\section{Introduction}
Discontinuous Galerkin (DG) methods are extensions of the usual 
(continuous) Galer-kin finite element methods.
The idea of introducing penalty terms in finite 
element methods originated from Nitsche \cite{Nitsche} and Babu\v{s}ka
\cite{Babuska}, while the idea of using discontinuous elements with
an interior penalty was introduced by Wheeler \cite{Wheeler}.  Later,
this approach was extended to the cases of nonlinear elliptic
and parabolic problems by Arnold \cite{Arnold}.   Since then,
the DG methods have developed in many directions.  For an account
of the history of DG methods for elliptic problems, readers are referred
to \cite[Section~2]{ABCM}.  Overall,
the mathematical theory of DG methods is well established
\cite{BrennerScott}, \revise{\cite{CDGH},} 
\cite{DiPietroErn}, \cite{ErnGuermond}, \cite{Riviere}.
In this paper, we consider the symmetric interior penalty (SIP) DG 
method, which is one of the most basic and well-known DG schemes.

The above mentioned papers and textbooks confirm that a
\textbf{shape-regularity condition} has always been imposed on meshes
for theoretical error analysis of DG methods.
The shape-regularity condition requires that elements in the
meshes must be neither very ``flat'' nor ``degenerated'' (see
Definition~\ref{def1}). If a mesh contains very flat elements, it is
said to be \textbf{anisotropic}.  The simple numerical experiment
described in Section~\ref{sec3.3} shows that the
standard SIP-DG method is not robust on anisotropic meshes, and 
that the shape-regularity condition is crucial for practical
computations.

The purpose of this paper is to introduce a new SIP-DG scheme that is
robust on anisotropic meshes.
To this end, we use the \textit{general} trace inequality to define a
new penalty term for the proposed SIP-DG scheme in Section~4.  In Section~5,
we show that the new scheme inherits all of the good properties
of standard SIP-DG methods.  That is, if the penalty
parameter is sufficiently large, the new SIP-DG scheme is
consistent, coercive, stable, and bounded on arbitrary (proper) meshes
(Lemma~$3^{new}$ and Lemma~\ref{lem8}). 
Those properties immediately yield error estimations of the new
SIP-DG scheme without imposing the shape-regularity condition (see
 Corollary~\ref{cor11}).  An immediate consequence is the
error estimate of order $\mathcal{O}(h^k)$ under
the maximum angle condition on meshes (see Corollary~\ref{cor12}). 
In Section~\ref{num-experi2}, we present the results of numerical
experiments to confirm the theoretical results obtained.  From the
results, we conclude that the newly presented SIP-DG scheme is robust on
anisotropic meshes.

\section{Preliminaries}
\subsection{The model problem}
Let $\Omega \subset \R^d$, ($d = 2, 3$) be a bounded polyhedral domain.
\revise{Let $L^2(\Omega)$, $H^1(\Omega)$, and $H_0^1(\Omega)$ be the 
usual Lebesgue and Sobolev spaces (see Section~\ref{functionspace}
for notation of their (semi)norms and inner products).}
We consider the following Poisson problem: find $u \in H^1(\Omega)$
such that
\begin{align}
   - \Delta u = \phi \quad \text{ in } \Omega, \qquad
   u = 0 \; \text{ on } \partial \Omega,
  \label{model-eq}
\end{align}
where $\phi \in L^2(\Omega)$ is a given function.
The weak form of the model problem is as follows:
\begin{align}
  a(u,v) := \int_\Omega \nabla u \cdot \nabla v \, \dd \bfx
  = (\phi, v)_\Omega, \quad \forall v \in H_0^1(\Omega).
   \label{bilinearform}
\end{align}
The model problem \eqref{model-eq} is said to satisfy
\textit{elliptic regularity} if there exists a positive constant
$C_{ell}$
such that the following \textit{a priori} estimate holds for the
exact solution $u$: 
\begin{align}
   \|u\|_{2,\Omega} \le C_{ell} \|\phi\|_{0,\Omega}, 
   \qquad \forall \phi \in L^2(\Omega).
    \label{elli-reg}
\end{align}
It is well known that if $\Omega$ is convex, then
the model problem \eqref{model-eq} satisfies elliptic
regularity \cite{Grisvard}.

\subsection{Meshes of $\Omega$}\label{sect:meshes}
Although DG methods allow elements with a variety of geometries,
we consider only simplicial elements in this paper.  Thus,
we suppose that the domain $\Omega$ is divided into a finite number
of triangles ($d=2$) or tetrahedrons ($d=3$), which are assumed to be
closed sets.  A mesh (or triangulation) of $\Omega$ is denoted by
$\T_h$.  That is, $\T_h$ is a finite set of triangles or tetrahedrons
 that has the following properties:
\begin{align*}
   \overline{\Omega} = \bigcup_{T \in \T_h} T, \qquad
   \revise{\operatorname{int} T_1}    \cap
    \revise{\operatorname{int} T_2} = \emptyset
   \;  \text{ if } \; T_1 \neq T_2,
\end{align*}
where \revise{$\operatorname{int} T_i$} is the interior of $T_i$.
For $T \in \T_h$, let $\bnormal_T$ be the unit outer normal vector
 on $\partial T$. 

In this paper, we assume that meshes are proper (or \textit{face-to-face}).
This means that, for $T_1, T_2 \in \T_h$, $T_1 \neq T_2$,
\begin{align}
 \text{
 if $T_1 \cap T_2 \neq \emptyset$,
then $T_1 \cap T_2$ is a $r$-face of $T_i$ $(i = 1, 2)$
with $0 \le r \le d-1$.
  } \label{proper-mesh}
\end{align}
For a $d$-simplex $T$ and its facet $f$, their Lebesgue and Hausdorff
measures are denoted by $|T|$ and $|f|$, respectively.

For $T$, we define $h_T := \revise{\operatorname{diam} T}$.  If $d=2$, $R_T$ is
the circumradius of $T$.  Note that $R_T = l_1l_2h_T/(4|T|)$, where
$l_1 \le l_2 \le h_T$ are the lengths of the edges of $T$. 
If $d=3$, $R_T$ is defined by
\begin{align}
  R_T := \frac{l_1l_2}{|T|}h_T^2,
  \label{def-RT}
\end{align}
where $l_1 \le l_2 \le \cdots \le h_T$ are the lengths of the edges of
$T$.  As has been seen in
\cite{IshiKobaTsuchi1,KobayashiTsuchiya1,KobayashiTsuchiya2,%
KobayashiTsuchiya3,KobayashiTsuchiya5,KobayashiTsuchiya6},
$R_T$ is an important parameter in measuring
interpolation errors on $d$-simplices.  For example, the errors of
Lagrange interpolation on $T$ are bounded in terms of $R_T$, as
presented by \eqref{L-est1} and \eqref{L-est2}. 

\vspace{3mm}
\noindent
\textit{Remark.}
 The \textit{best} definition of $R_T$ for tetrahedrons remains
an open problem.  A simple example immediately rejects the idea
that $R_T$ for a tetrahedron might be the radius of its circumsphere
\cite[p.~3]{KobayashiTsuchiya5}.    The definition 
\eqref{def-RT} is given in \cite{IshiKobaTsuchi1}. In
\cite{IshiKobaTsuchi1}, another parameter, denoted by $H_T$,
is introduced, and it is shown that $R_T$ and $H_T$ are equivalent
\revise{(see also \cite{KobayashiTsuchiya6})}. 
In \cite{KobayashiTsuchiya5}, the \textit{projected circumradius}
of $T$ is defined for a tetrahedrons.  It is conjectured that
$R_T$ and the projected circumradius are equivalent.

\vspace{3mm}
Let $\E_h := \{f \mid f \text{ is a facet of } T \in \T_h\}$.
That is, $\E_h$ is the set of all edges ($d=2$) or faces ($d=3$)
in $\T_h$. Then, let 
$\E_h^\partial := \{f \in \E_h \mid f \subset \partial\Omega\}$ and
$\E_h^o := \E_h \backslash \E_h^\partial$.

Suppose that we consider a (possibly infinite) family of meshes
$\{\T_h\}_{h > 0}$ with $h \to 0$.
\begin{definition}\label{def1}
(1) The family of meshes $\{\T_h\}_{h > 0}$ is said to satisfy 
the \textbf{shape-regularity condition} with respect to $\sigma$
if there exists a positive constant $\sigma$ such that
\begin{align*}
   \frac{h_T}{\rho_T} \le \sigma, \qquad 
   \forall T \in \T_h, \quad \forall h > 0,
\end{align*}
where 
 $\rho_T$ is the diameter of the inscribed ball of $T$.  We call
 $\sigma$ the \textit{shape-regular constant} in this paper.

(2) The family of meshes $\{\T_h\}_{h > 0}$ is said to satisfy 
the \textbf{maximum angle condition} with respect to a constant
$C_{max}$ $(\pi/3 < C_{max} < \pi)$
if the following hold for all $T \in \T_h$ and for all $\T_h$:
\begin{itemize}
 \item An arbitrary inner angle $\theta$ of $T$ is $\theta \le C_{max}$
       ($d=2$), or
 \item An arbitrary inner angle $\theta$ of any facet of $T$ is
$\theta \le C_{max}$, and an arbitrary dihedral angle $\eta$ of $T$ is
$\eta \le C_{max}$ ($d=3$).
\end{itemize}

(3) The family of meshes $\{\T_h\}_{h > 0}$ is said to satisfy 
the \textbf{circumradius condition} if the family satisfies
\begin{align*}
  \lim_{h \to 0}\max_{T \in \T_h} R_T = 0.
\end{align*}

\vspace{3mm}
In this paper, we always assume that the family $\{\T_h\}_{h > 0}$ of
meshes satisfies the circumradius condition.
\end{definition}

\noindent
\textit{Remark.}  If we deal with only a finite family of meshes
 $\{\T_h\}$ and we take a sufficiently large $\sigma > 0$, then
the family satisfies the shape-regularity condition because $\{\T_h\}$
contains only a finite number of $d$-simplices.  However,
if $\sigma$ is too large (say, $\sigma \ge 10$), we commonly say that
$\{\T_h\}$ is \textit{not} shape-regular.  In such a case, as
mentioned in Section~1,  $\{\T_h\}$ is said to be \textbf{anisotropic}. 

\subsection{Function spaces}\label{functionspace}
Let $k \ge 1$ be a positive integer.  We use the notation $L^2(\Omega)$,
$L^2(f)$ $(f \in \E_h)$, $H^k(\Omega)$,
$H_0^1(\Omega)$ for the usual Lebesgue and Sobolev spaces.
We denote their norms and
semi-norms by, for example,  $\|\cdot\|_{0,f}$, $|\cdot|_{1,\Omega}$,
and their inner products by $(\cdot,\cdot)_f$, $(\cdot,\cdot)_\Omega$.
Let $\PP_k(K)$ be the set of all polynomials defined on
the closed set $K$ with degree less than or equal to a positive
integer $k$.

As usual,  we introduce the broken Sobolev and polynomial spaces by
\begin{align*}
   \dsp H^k(\T_h) &:= \left\{v \in L^2(\Omega) \bigm|
    v|_{T} \in H^k(T), \forall T \in \T_h \right\}, \\
   \PP_k(\T_h) & := \left\{v \in L^2(\Omega) \bigm|
    v|_{T} \in \PP_k(T), \forall T \in \T_h \right\}.
\end{align*}
Finally, define
 $V := H_0^1(\Omega)$, $V_* := H_0^1(\Omega) \cap H^2(\Omega)$,
$V_h := \PP_k(\T_h)$,  and $V_{*h} := V_* + V_h$.

\subsection{Jump and mean of functions on $f \in \E_h$}
For each interior facet $f \in \E_h^o$, there are two simplices
that share $f$.  We number those simplices as $T_{f}^{i} \in \T_h$
 $(i=1,2)$ and fix the numbering once the mesh is obtained.
Then, we have $f = T_{f}^{1} \cap T_{f}^{2}$.
\revise{
Recalling that $\bnormal_T$ is the unit outer nomal vector on
$\partial T$, define $\bnormal_f := \bnormal_{T_{f}^{1}}$.}
%
\, For $v \in H^2(\T_h)$, we set $v_1 := v|_{T_{f}^{1}}$ and
$v_2 := v|_{T_{f}^{2}}$.  We denote the trace operator on $T_f^i$
to $f$ by $\gamma_f^i$ ($i=1,2$).  Define
\begin{align*}
   [v] := \gamma_f^1(v_1) - \gamma_f^2(v_2),
   \qquad  \{v\} := \frac{1}{2}\left(
   \gamma_f^1(v_1) + \gamma_f^2(v_2)\right).
\end{align*}
The jump $[\nabla v]_f$ and average $\{\nabla v\}_f$ are
defined in a similar way. 

If $g \in \E_h^\partial$, then $g \subset \partial\Omega$.
Let $g \subset \partial T_g$ with $T_g \in \T_h$.  Then, define
\begin{align*}
   [v] = \{v\} := \gamma_g(v|_{T_{g}}).
\end{align*}


\section{Standard SIP-DG scheme}
\subsection{Definition of SIP-DG scheme}
In the SIP-DG scheme, the bilinear form $a(u,v)$ in 
\eqref{bilinearform} is discretized as
\begin{align*}
  \asip(v,w_h) & := \sum_{T \in \T_h} \int_T \nabla v \cdot 
   \nabla w_h \dd \bfx - \sum_{f \in \E_h} \int_{f} [w_h] 
 \{\nabla v\} \cdot \bnormal_f \dd s  \\
 & \quad - \sum_{f \in \E_h} \int_{f} 
 [v] \{\nabla w_h\} \cdot \bnormal_f \dd s +
     \eta \sum_{f \in \E_h} \frac{1}{h_f} \int_f [v][w_h]\dd s
\end{align*}
for $v \in V_{*h}$ and $w_h \in V_h$, where 
$h_f := \revise{\operatorname{diam} f}$.
Here, $\eta$ is a \textbf{penalty parameter} that is taken
to be sufficiently large.  To make the notation concise, we set
\begin{align}
  a_h^{(0)}(v,w_h) & := \sum_{T \in \T_h} \int_T \nabla v \cdot 
   \nabla w_h \, \dd \bfx \notag \\
  J_h(v,w_h) & := \sum_{f \in \E_h} \int_{f} [w_h] 
 \{\nabla v\} \cdot \bnormal_f \, \dd s + \sum_{f \in \E_h} \int_{f} 
 [v] \{\nabla w_h\} \cdot \bnormal_f \, \dd s, \notag \\
  P_h^{std}(v,w_h) & :=  \eta \sum_{f \in \E_h}
   \frac{1}{h_f} \int_f [v][w_h] \, \dd s. \label{penalty-term}
\end{align}
The terms $J_h(v,w_h)$ and $P_h^{std}(v,w_h)$ are called the
\textit{jump term} and \textit{penalty term}, respectively.
The discretized bilinear form $\asip(v,w_h)$ is written as
\begin{align*}
   \asip(v,w_h) = a_h^{(0)}(v,w_h) - J_h(v,w_h) + P_h^{std}(v,w_h).
\end{align*}

\begin{definition}
The SIP-DG 
scheme for the model problem is defined as follows: find
$u_h \in V_h$ such that
\begin{align}
   \asip(u_h, v_h) = (\phi, v_h)_\Omega, \quad \forall v_h \in V_h.
   \label{sip-dg}
\end{align}
\end{definition}

\subsection{Properties of SIP-DG scheme and error analysis}
In the following, we summarize the properties of the SIP-DG method.
For their proofs, readers are referred to the standard textbooks
\cite{BrennerScott}, \cite{DiPietroErn}, \cite{ErnGuermond}, 
\cite{Riviere}.  In this section, we mainly refer to \cite{DiPietroErn}.

\begin{lemma}[Consistency]\cite[Lemma~4.8]{DiPietroErn}
\label{thm-consist}
\quad
The exact solution $u \in V_*$ of the model problem
\eqref{model-eq} is consistent:
\begin{align*}
  \asip(u,v_h) = (\phi, v_h)_\Omega, \quad \forall v_h \in V_h.
\end{align*}
Therefore, the solution $u_h \in V_h$ of the SIP-DG method 
\eqref{sip-dg} satisfies the Galerkin orthogonality:
\begin{align*}
  \asip(u - u_h,v_h) = 0, \quad \forall v_h \in V_h.
\end{align*}
\end{lemma}

We define the norms associated with the bilinear form $\asip$ as:
\begin{align*}
  \|v\|_{\DG} & := \left(a_h^{(0)}(v,v)
  + P_h^{std}(v,v) \right)^{1/2}, \qquad  v \in V_{*h}, \\
  \|v\|_{\DG*} & := \left(\|v\|_{\DG}^2 +
   \eta^{-1} \sum_{f \in \E_h} h_f
     \|\{\nabla v\}\cdot\bnormal_f\|_{0,f}^2 \right)^{1/2}.
\end{align*}

\begin{lemma}
Suppose that the mesh $\T_h$ is shape-regular with respect to a constant
$\sigma > 0$ and the penalty parameter $\eta$ is sufficiently large.
Then,
\begin{itemize}
 \item[$(1)$] $(\mathbf{Discrete\; coercivity})$
   \cite[Lemma~4.12]{DiPietroErn} \quad
   The bilinear form $\asip$ is coercive in $V_h$ with
  respect to the norm $\|\cdot\|_{\DG}$:
 \begin{align*}
    \asip(w_h,w_h) \ge \frac{1}{2} \|w_h\|_{\DG}^2, \qquad
     \forall w_h \in V_h.
 \end{align*}
 \item[$(2)$] $(\mathbf{Discrete\; stability})$
 The following inequality holds:
\begin{align*}
 \frac{1}{2}\|v_h\|_{\DG} \le \sup_{w_h \in V_h} 
   \frac{\asip(v_h,w_h)}{\|w_h\|_{\DG}}, \qquad
   \forall v_h \in V_h.
\end{align*}
\item[$(3)$] $(\mathbf{Boundedness})$
  \cite[Lemma~4.16]{DiPietroErn} \quad
The following inequality holds:
\begin{align*}
 \asip(v,w_h) \le C \|v\|_{\DG*}\|w_h\|_{\DG}, 
  \qquad \forall (v,w_h) \in V_{*h}\times V_h,
\end{align*} 
where the constant $C := C(\eta,\sigma)$ is independent of $h$.
\end{itemize}
\end{lemma}

\begin{theorem}\label{error-est-old}
\cite[Theorem~4.17]{DiPietroErn} \quad
Suppose that the mesh $\T_h$ is shape-regular with respect to a constant
$\sigma > 0$ and the penalty parameter $\eta$ is sufficiently large.
Then, there exists a unique SIP-DG solution $u_h \in V_h$ of 
\eqref{sip-dg}, and the following error estimate holds:
\begin{align*}
   \|u - u_h\|_{\DG} \le C \inf_{y_h \in V_h} \|u - y_h\|_{\DG*},
\end{align*}
where the constant $C$ depends only on the penalty parameter $\eta$
and $\sigma$.
\end{theorem}

\begin{corollary}
\cite[Corollary~4.18]{DiPietroErn} \quad
Suppose that the assumptions of Theorem~\ref{error-est-old} hold
and that the exact solution $u$ of the model problem
\eqref{model-eq} belongs to $H^2(\Omega)$.  Then, we have the
following error estimate:
\begin{align*}
   \|u-u_h\|_{\DG} \le C h |u|_{2,\Omega},
\end{align*}
where the constant $C$ depends on $\eta$ and $\sigma$, but
is independent of $h$.
\end{corollary}

\subsection{Numerical experiments (part 1)}\label{sec3.3}
We consider a numerical experiment to examine how the shape-regular
constant $\sigma$ affects the practical computations involved in the
standard SIP-DG scheme.

Set $\Omega:=(0,1)\times(0,1)$ and
$\phi(x,y):= \pi^2\sin(\pi x)\sin(\pi y)$ in the model problem
\eqref{model-eq}.  Then, the exact solution is
$u(x,y)=\sin(\pi x)\sin(\pi y)/2$.
Let $n$ and $m$ be positive integers.  We divide the horizontal
and vertical sides of $\Omega$ into $n$ and $m$ equal segments,
respectively.  We then draw diagonal lines in each small rectangle to
define the mesh, as depicted in Figure~\ref{fig1}.
\begin{figure}[ht]
 \begin{minipage}[c]{6cm}
  \includegraphics[width=5cm]{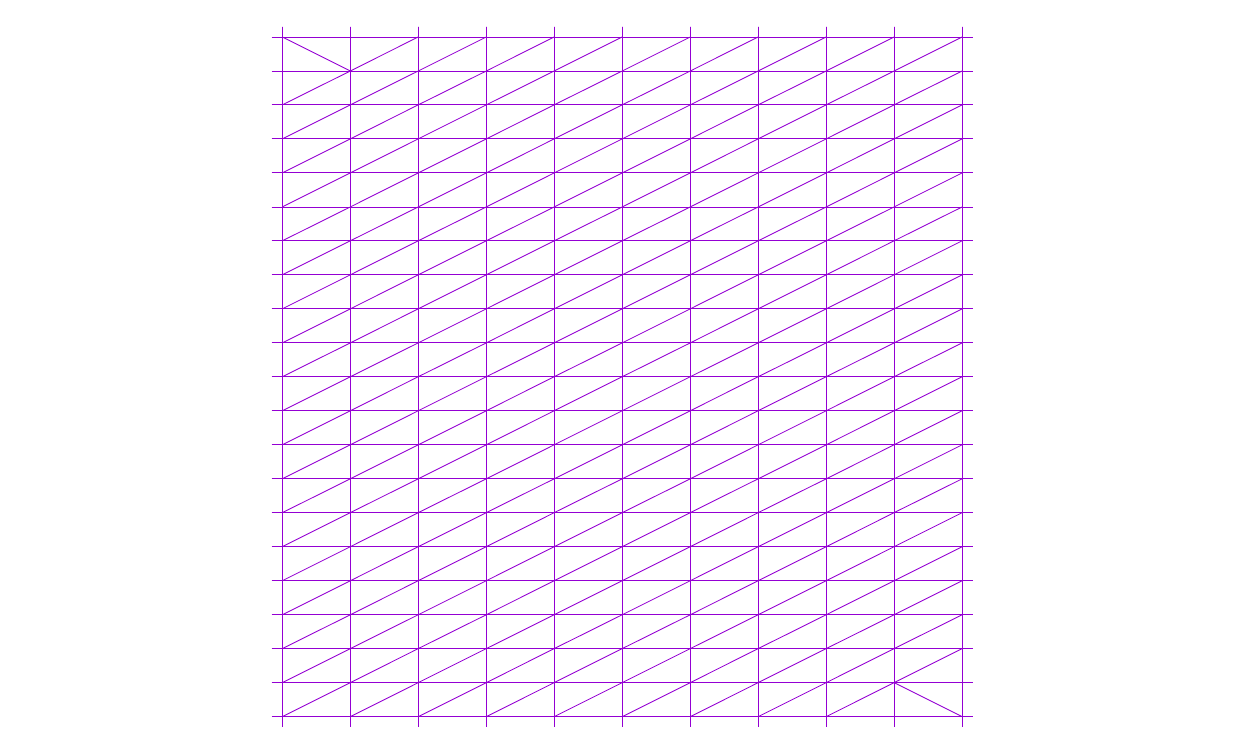} 
\end{minipage}
\caption{Mesh constructed in $\Omega$. $n=10$, $m=20$.}
\label{fig1}
\end{figure}

We fix $n=40$ and the penalty parameter $\eta = 10$.  We apply
the standard SIP-DG method to the model problem with various $m$.
The conjugate gradient method with the incomplete Cholesky decomposition
preconditioner (ICCG) is used for the linear solver. 
The successive over-relaxation (SOR) method is also used
occasionally to check whether the obtained $u_h$ is reasonable. The
results are summarized in Table~\ref{tab1}.
\begin{table}[htbp]
\caption{Errors produced by the SIP-DG method.} \label{tab1}
\begin{tabular}[t]{|c|c|c|c|c|c|c|}
\hline
 $m$ & $h$ & $R$ & $L^2$-error & $H^1(\T_h)$-error & $P_h^{std}$-error & DG-error \\
\hline 
40 & 3.536e-2 & 1.768e-2 & 2.991e-4 & 3.643e-2 &  
\revise{1.972e-2} & \revise{4.142e-2} \\
\hline
60 & 3.006e-2 & 1.503e-2 & 2.095e-4 & 3.067e-2 &
\revise{1.710e-2} & \revise{3.511e-2} \\
\hline 
80 & 2.795e-2 & 1.398e-2 & 1.706e-4 & 2.807e-2 &
\revise{1.645e-2} & \revise{3.253e-2} \\
\hline 
100 & 2.693e-2 & 1.346e-2 & 1.484e-4 & 2.663e-2 &
\revise{1.634e-2} & \revise{3.124e-2} \\
\hline 
120 & 2.637e-2 & 1.318e-2 & 1.334e-4 & 2.581e-2 &
\revise{1.648e-2} & \revise{3.062e-2} \\
\hline 
140 & 2.602e-2 & 1.301e-2 & 1.222e-4 & 2.564e-2 &
\revise{1.720e-2} & \revise{3.088e-2} \\
\hline 
160 & 2.578e-2 & 1.289e-2 & 1.147e-4 & 3.173e-2 &
\revise{2.620e-2} & \revise{4.115e-2} \\
\hline 
180 & 2.562e-2 & 1.281e-2 & 1.150e-4 & 3.330e-2 &
\revise{2.858e-2} & \revise{4.389e-2} \\
\hline 
200 & 2.550e-2 & 1.275e-2 & - & - & - & -\\
\hline 
\end{tabular}
\end{table}

Here, for $T \in \T_h$, 
$h=\revise{\operatorname{diam}{T}}$, $R$ is the circumradius of $T$, 
the ``$L^2$-error'' is $|u - u_h|_{L^2(\Omega)}$,
the ``$H^1(\T_h)$-error'' is $a_h^{(0)}(u - u_h,u-u_h)^{1/2}$,
the ``$P_h^{std}$-error'' is $P_h^{std}(u-u_h,u-u_h)^{1/2}$, and
the ``DG-error'' is $\|u-u_h\|_{\DG}$.
\revise{
We employ the 4-point Gauss quadrature of degree 3  on triangles to
compute those errors.
}
We see that the errors given by the  SIP-DG method decrease as
$m$ increases until $m = 100$, which is consistent with
the theoretical error estimates.  However,
the errors increase as $m$ increases from $m=120$.
The ICCG iterations do not converge for $m = 200$,
while the SOR iterations give almost the same results until $m=160$.
For $m=180, 200$, the SOR iterations converge quickly
but the obtained $u_h$ are not reasonable.

We also examined the case $m=400$.  In this case,
we required $\eta=30$ to obtain reasonable $u_h$.

\section{New penalty term and SIP-DG scheme}
From the numerical experiments described in the previous section,
we can conclude that the shape-regularity condition is crucial for 
the standard SIP-DG method with a fixed penalty parameter $\eta$.
It is natural to wonder why this is the case.

The most important term in the SIP-DG scheme is the penalty term
$P_h^{std}(v,w_h)$ defined by \eqref{penalty-term}.
The penalty term originates from the trace inequalities
\begin{align}
  \|v\|_{0,f} & \le  \frac{C_1^{tr}}{h_f^{1/2}} \|v\|_{0,T}, \quad 
  \|\nabla v\cdot\bnormal\|_{0,f} \le \frac{C_2^{tr}}{h_f^{1/2}} 
  \|\nabla v\|_{0,T}, \quad \forall v \in \PP_k(T),
    \label{trace-in2}
\end{align}
where $f$ is an arbitrary facet of $T \in \T_h$.  Note that the
constants $C_i^{tr}$ ($i=1,2$) strongly depend on the shape-regular
constant $\sigma$.

To avoid imposing the shape-regularity condition, we
adopt the \textit{general} trace inequalities
\begin{align}
     \|v\|_{0,f} & \le C_3^{tr} \frac{|f|^{1/2}}{|T|^{1/2}} 
     \|v\|_{0,T}, \quad
    \left\| \nabla v \cdot \bnormal \right\|_{0,f}
     \le C_4^{tr} \frac{|f|^{1/2}}{|T|^{1/2}} 
     \|\nabla v\|_{0,T}, \quad \forall v \in \PP_k(T),
    \label{trace-in1}
\end{align}
which are valid on an arbitrary $d$-simplex $T$.
Warburton and Hesthaven \cite{WarHest} presented explicit forms 
of the constants $C_i^{tr}$ ($i=3,4$) that are
independent of the geometry of $T$.
Note that \eqref{trace-in2} is a ``simplified'' version of
\eqref{trace-in1} under the shape-regularity condition.
We introduce a quantity on each $f \in \E_h$ below.

Let $f \in \E_h^o$.  Then, there exist $T_f^{1}$, $T_f^{2} \in \T_h$
such that $f = T_f^{1} \cap T_f^{2}$.   Let
$\tT_{f}^{i}$ be the $d$-simplex whose vertices are those of
$f$ and the barycenter of $T_{f}^{i}$ ($i=1,2$).  Then, define
\begin{align*}
   \jumpterm :=  \frac{|f|}{|\tT_{f}^{1}|}
     + \frac{|f|}{|\tT_{f}^{2}|}.
\end{align*}
If $g \in \E_h^\partial$, then $g \subset \partial\Omega$.
Let $g \subset \partial T_g$ with $T_g \in \T_h$ and $\tT_g$ be the
$d$-simplex whose vertices are those of $g$ and the barycenter of
$T_g$. Define
\begin{align*}
   \left\{\frac{|g|}{|\tT_g|}\right\} := \frac{|g|}{|\tT_g|}.
\end{align*}
Note that
\begin{align}
   \left(\bigcup_{g \in \E_h^\partial}\tT_g\right) \cup \left(
   \bigcup_{f \in \E_h^o} \tT_{f}^{1}\cup
    \tT_{f}^{2}\right) = \overline{\Omega}.
   \label{covering}
\end{align}
See Figure~\ref{fig2}.  We remark that the only information we need for
the simplex $\tT_f^i$ is its measure $|\tT_f^i| = |T_f^i|/(d+1)$.

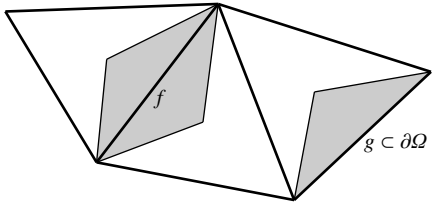
\begin{figure}[htbp]
\begin{tikzpicture}[line width = 1pt]
  \coordinate(A) at (-2.6,2);
  \coordinate(B) at (-1.4,0);
  \coordinate(C) at (1.2,-0.5);
  \coordinate(D) at (3,1.2);
  \coordinate(E) at (0.2,2.1);
  \coordinate(F) at ($(A)!0.5!(B)$);
  \coordinate(G) at ($(E)!0.6666666!(F)$);
  \coordinate(H) at ($(B)!0.5!(C)$);
  \coordinate(I) at ($(E)!0.6666666!(H)$);
  \coordinate(J) at ($(D)!0.5!(C)$);
  \coordinate(K) at ($(E)!0.6666666!(J)$);
 \fill [black!20] (B)--(E)--(G)--(B)--cycle;
 \fill [black!20] (B)--(E)--(I)--(B)--cycle;
 \fill [black!20] (C)--(D)--(K)--(C)--cycle;
  \draw[line width=1pt] (A) -> (B);
  \draw[line width=1pt] (B) -> (C);
  \draw[line width=1pt] (C) -> (D) node[pos=0.45,right]{$g \subset\partial\Omega$};
  \draw[line width=1pt] (D) -> (E);
  \draw[line width=1pt] (E) -> (A);
  \draw[line width=1pt] (E) -> (B) node[pos=0.6,right]{$f$};
  \draw[line width=1pt] (E) -> (C);
  \draw[line width=0.5pt] (E) -> (G);
  \draw[line width=0.5pt] (B) -> (G);
  \draw[line width=0.5pt] (E) -> (I);
  \draw[line width=0.5pt] (B) -> (I);
  \draw[line width=0.5pt] (C) -> (K);
  \draw[line width=0.5pt] (D) -> (K);
\end{tikzpicture}
\caption{Domains for the new penalty terms.} \label{fig2}
\end{figure}

With the quantity $\jumpterm$ defined for $f \in \E_h$, we (re)define
the penalty term and SIP-DG bilinear form $a_h^{std}$ as
\begin{align}
 P_h^{new}(v,w_h) & := \eta \sum_{f \in \E_h}
   \jumpterm \int_f [v] [w_h] \dd s, \label{newpenalty} \\
  \asipnew(v,w_h) & := a_h^{(0)}(v,w_h) - J_h(v,w_h) + P_h^{new}(v,w_h).
  \label{newa}
\end{align}
Note that, if a mesh satisfies the shape-regularity condition,
$\jumpterm \approx \frac{1}{h_f}$ for any facet $f \in \E_h$, and
$P_h^{new}$ becomes equivalent to the standard penalty term $P_h^{std}$.

\begin{definition}
The SIP-DG scheme for the model problem is redefined as follows: find
$u_h \in V_h$ such that
\begin{align}
   \asipnew(u_h, v_h) = (\phi, v_h)_\Omega, \quad \forall v_h \in V_h.
   \label{sip-dg-new}
\end{align}
\end{definition}

\noindent
\textit{Remark.} To see the meaning of the new penalty terms
\eqref{newpenalty}, let us consider a small part of the mesh
of Figure~\ref{fig1}, as shown in Figure~\ref{fig3}.

\newcommand{\mayu}{4}
\newcommand{\koyume}{0.8}
\begin{figure}[htbp]
\begin{tikzpicture}[line width = 1pt]
  \coordinate(A) at (-\mayu,0);
  \coordinate(B) at (0,0);
  \coordinate(C) at (0,\koyume);
  \coordinate(D) at (\mayu,\koyume);
  \coordinate(E) at (-\mayu,-\koyume);
  \draw[line width=1pt] (A) -> (B) node[pos=0.9,above]{$f_1$};
  \draw[line width=1pt] (B) -> (C) node[pos=0.75,right]{$f_2$};
  \draw[line width=1pt] (C) -> (A) node[pos=0.4,above]{$T$};
  \draw[line width=0.5pt] (E) -> (D);
  \draw[line width=0.5pt] (E) -> (A);
  \draw[line width=0.5pt] (C) -> (D);
\end{tikzpicture}
\caption{Part of the mesh.} \label{fig3}
\end{figure}
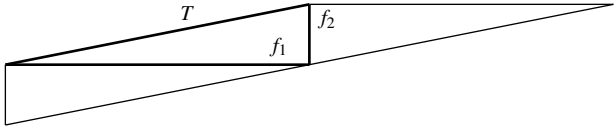
The mesh consists of congruent right triangles. Let $T$ be one of the
right triangles, and $f_1$, $f_2$ be its edges as depicted in
Figure~\ref{fig3}.  Set $h = |f_1|$.  Then, $|f_2| = \alpha h$, where
$0 < \alpha < 1$. Elementary geometry tells us that
$\left\{\frac{f_2}{|\tT_{f_2}|}\right\} = 12/h$,
\begin{align}
   \left\{\frac{|f_1|}{|\tT_{f_1}|}\right\}
   = \frac{12}{\alpha h}, \quad \text{ and }\quad
  \eta \left\{\frac{|f_1|}{|\tT_{f_1}|}\right\} \int_{f_1}
   [v][w_h] \dd s = \frac{12\eta}{\alpha h} \int_{f_1} [v][w_h] \dd s.
   \label{adaptive}
\end{align}
Note that if $T$ is becoming degenerated,
the coefficient $12\eta/\alpha$ will increase.
This means that we can interpret the coefficient $\eta\jumpterm$ as an
``improved'' version of the standard penalty term coefficient
$\frac{\eta}{h_f}$ with the ``adaptive'' parameter $\alpha$.
 $\square$

\section{Properties of the new bilinear form $\asipnew$ and
error estimates}
In this section, we show that, without imposing 
the shape-regularity condition, the new SIP-DG scheme
inherits all of the good properties from the standard SIP-DG scheme.

From the definitions, it is clear that Lemma~\ref{thm-consist} still
holds for the new $\asipnew$:
\begin{miho}[Consistency] \label{thm-consist2}
The exact solution $u \in V_*$ of the model problem
\eqref{model-eq} is consistent:
\begin{align*}
  \asipnew(u,v_h) = (\phi,v_h)_\Omega, \quad \forall v_h \in V_h.
\end{align*}
Therefore, the solution $u_h \in V_h$ of SIP-DG method 
\eqref{sip-dg} satisfies the Galerkin orthogonality:
\begin{align*}
  \asipnew(u - u_h,v_h) = 0, \quad \forall v_h \in V_h.
\end{align*}
\end{miho}

We redefine the norms associated with the SIP-DG scheme:
\begin{align}
  \|v\|_{\DG}^{new} & := \left(a_h^{(0)}(v,v)
  + P_h^{new}(v,v) \right)^{1/2}, \qquad  v \in V_{*h}, \\
  \|v\|_{\DG*}^{new} & := \left(
   \revise{\left(\|v\|_{\DG}^{new}\right)^2} +
   \eta^{-1} \sum_{f \in \E_h} \jumpterm^{-1}
     \|\{\nabla v\}\cdot\bnormal_f\|_{0,f}^2 \right)^{1/2}.
\end{align}

The following lemma holds.  Its proof is quite similar to 
the standard proofs.  Here, we loosely follow the proofs given by
Di Pietro and Ern \cite{DiPietroErn}.
\begin{lemma}\label{lem8}
Suppose that the penalty parameter $\eta$ is sufficiently large. Then,
\begin{itemize}
 \item[$(1)$] $(\mathbf{Discrete\; coercivity})$
   The bilinear form $\asipnew$ is coercive in $V_h$ with
  respect to the norm $\|\cdot\|_{\DG}^{new}$:
\vspace{-3mm}
 \begin{align*}
    \asipnew(w_h,w_h) \ge \frac{1}{2}
  \left(\|w_h\|_{\DG}^{new}\right)^2,
    \qquad \forall w_h \in V_h.
 \end{align*}
 \item[$(2)$] $(\mathbf{Discrete\; stability})$
 The following inequality holds:
\vspace{-2mm}
\begin{align*}
 \frac{1}{2}\|v_h\|_{\DG}^{new} \le \sup_{w_h \in V_h} 
   \frac{\asipnew(v_h,w_h)}{\|w_h\|_{\DG}^{new}}, \qquad
   \forall v_h \in V_h.
\end{align*}
\item[$(3)$] $(\mathbf{Boundedness})$
The following inequalities hold:
\begin{align}
 \asipnew(v,w_h) & \le C \|v\|_{\DG*}^{new}\|w_h\|_{\DG}^{new}, 
  \qquad \forall (v,w_h) \in V_{*h}\times V_h, 
    \label{bounded1} \\
  \asipnew(v,w) & \le \|v\|_{\DG*}^{new}\|w\|_{\DG*}^{new}, 
  \qquad \forall (v,w) \in V_{*h}\times V_{*h},
  \label{bounded2}
\end{align} 
where the constant $C := C(\eta, C_4^{tr})$ is independent of $h$
and the geometry of elements in $\T_h$.
\end{itemize}
\end{lemma}
\begin{proof}
(1) Let $w_h \in V_h$. For $f \in \E_h^o$, there exist
$T_f^{i} \in \E_h$ ($i=1,2$) with
$f = T_f^{1} \cap T_f^{2}$.
It follows from the trace inequality \eqref{trace-in1} that
\begin{align*}
   \left\| \left\{\nabla w_h \right\} \cdot \bnormal_f \right\|_{0,f}
  & \le C_4^{tr} \frac{|f|^{1/2}}{2} \left(
  \frac{\|\nabla w_h\|_{0,\tT_f^{1}}}{|\tT_f^{1}|^{1/2}}  
  + \frac{\|\nabla w_h\|_{0,\tT_f^{2}}}{|\tT_f^{2}|^{1/2}}  \right).
\end{align*}
Hence, we have
\begin{align}
 (\{\nabla w_h \} \cdot \bnormal_f, [w_h])_f & \le
  \|\{\nabla w_h \} \cdot \bnormal_f\|_{0,f}\|[w_h]\|_{0,f} \notag \\
 & \le  \frac{C_4^{tr} |f|^{1/2}}{2}
   \left(\frac{\|\nabla w_h\|_{0,\tT_f^{1}}}{|\tT_f^{1}|^{1/2}} + 
    \frac{\|\nabla w_h\|_{0,\tT_f^{2}}}{|\tT_f^{2}|^{1/2}} \right)
    \|[w_h]\|_{0,f} \notag \\
& \le \frac{C_4^{tr}}{2}
   \left(\|\nabla w_h\|_{0,\tT_f^{1}}^2 + \|\nabla w_h\|_{0,\tT_f^{2}}^2
   \right)^{1/2} \jumpterm^{1/2} \|[w_h]\|_{0,f}.
    \label{core4}
   \end{align}
We also obtain a similar inequality for the case $f \in \E_h^{\partial}$.

Because of $\displaystyle J_h(w_h,w_h) = 2\sum_{f \in \E_h}
(\{\nabla w_h \} \cdot \bnormal_f, [w_h])_f$ and \eqref{covering}, we
have that 
\begin{align} 
   J_h(w_h,w_h) & \le
   C_4^{tr}  \left(\sum_{f \in \E_h}
   \jumpterm  \|[w_h]\|_{0,f}^2 \right)^{1/2} 
   \left(\sum_{T \in \T_h} \|\nabla w_h\|_{0,T}^2 \right)^{1/2}.
   \label{core5}
\end{align}
Thus, it follows from the arithmetic-geometric mean that
\begin{align*}
    J_h(w_h,w_h) & \le
   \delta \sum_{T \in \T_h} \|\nabla w_h\|_{0,T}^2
  + \frac{(C_4^{tr})^2}{\delta} \sum_{f \in \E_h}
    \jumpterm  \|[w_h]\|_{0,f}^2  \\
 & = \delta a_h^{(0)}(w_h,w_h) + 
    \frac{(C_4^{tr})^2}{\revise{4} \hspace{0.3mm}
     \delta\eta} P_h^{new}(w_h,w_h)
\end{align*}
for some constant $\delta > 0$.  Set $\delta := 1/2$, and let $\eta$ be
sufficiently large so that
\revise{$\eta \ge (C_4^{tr})^2$}.
Then, the following coercivity holds:
\begin{align*}
 \asipnew(w_h,w_h) & = a_h^{(0)}(w_h,w_h) - J_h(w_h,w_h)
   + P_h^{new}(w_h,w_h) \\
 & \ge \frac{1}{2} a_h^{(0)}(w_h,w_h) + \frac{1}{2}P_h^{new}(w_h,w_h)
  = \frac{1}{2} (\|w_h\|_{\DG}^{new})^2.
\end{align*}

(2) For arbitrary $v_h \in V_h$, we have
\begin{align*}
   \frac{1}{2}\|v_h\|_{\DG}^{new} \le %
   \frac{\asipnew(v_h,v_h)}{\|v_h\|_{\DG}^{new}}
    \le \sup_{w_h \in V_h} 
  \frac{\asipnew(v_h,w_h)}{\|w_h\|_{\DG}^{new}}
\end{align*}
because of the coercivity.

(3) By the Cauchy--Schwarz inequality, we see
that
\begin{align*}
  a_h^{(0)}(v,w_h) & \le \left(a_h^{(0)}(v,v)\right)^{1/2}
     \left(a_h^{(0)}(w_h,w_h)\right)^{1/2}, \\
   P_h^{new}(v,w_h) & \le \left(P_h^{new}(v,v)\right)^{1/2}
     \left(P_h^{new}(w_h,w_h)\right)^{1/2}.
\end{align*}
Furthermore, we have that
\begin{align*}
  \sum_{f \in \E_h} (\{\nabla v\}\cdot\bnormal_f, [w_h])_f
   & \le \sum_{f \in \E_h}
   \eta^{-1/2}\jumpterm^{-1/2}
 \|\{\nabla v\}\cdot\bnormal_f\|_{0,f} \cdot
  \eta^{1/2}\jumpterm^{1/2}\|[w_h]\|_{0,f} \\
 & \le \left( \eta^{-1}\sum_{f \in \E_h}
  \jumpterm^{-1}\|\{\nabla v\}\cdot\bnormal_f\|_{0,f}^2
  \right)^{1/2}\left(P_h^{new}(w_h,w_h)\right)^{1/2}.
\end{align*}
\end{proof}
It follows from \eqref{core4} and \eqref{covering} that
\begin{align*}
   \sum_{f \in \E_h} (\{\nabla w_h\}\cdot\bnormal_f, [v])_f
  & \le \frac{C_4^{tr}}{2}  \sum_{f \in \E_h}
    \left(\|\nabla w_h\|_{0,T_1'}^2 + \|\nabla w_h\|_{0,T_2'}^2
   \right)^{1/2} \jumpterm^{1/2} \|[v]\|_{0,f} \\
  & \le \frac{C_4^{tr}}{2\eta} 
   \left(\sum_{T \in \T_h} \|\nabla w_h\|_{0,T}^2 \right)^{1/2}
    \left(P_h^{new}(v,v)\right)^{1/2}.
\end{align*}
Therefore, we obtain
\begin{align*}
 J_h(v,w_h) & =
  \sum_{f \in \E_h} (\{\nabla v\}\cdot\bnormal_f, [w_h])_f
  + \sum_{f \in \E_h} (\{\nabla w_h\}\cdot\bnormal_f, [v])_f \\
  & \le \left(\sum_{f \in \E_h}
  \eta^{-1}\jumpterm^{-1}\|\{\nabla v\}\cdot\bnormal_f\|_{0,f}^2
  \right)^{1/2}\left(P_h^{new}(w_h,w_h)\right)^{1/2} \\
 & \quad + \frac{C_4^{tr}}{2\eta}
    \left(a_h^{(0)}(w_h,w_h)\right)^{1/2}
    \left(P_h^{new}(v,v)\right)^{1/2}.
\end{align*}
For $(v,w) \in V_{*h} \times V_{*h}$, we immediately obtain
\begin{align*}
 J_h(v,w) & =
  \sum_{f \in \E_h} (\{\nabla v\}\cdot\bnormal_f, [w])_f
  + \sum_{f \in \E_h} (\{\nabla w\}\cdot\bnormal_f, [v])_f \\
  & \le \left(\sum_{f \in \E_h}
  \eta^{-1}\jumpterm^{-1}\|\{\nabla v\}\cdot\bnormal_f\|_{0,f}^2
  \right)^{1/2}\left(P_h^{new}(w,w)\right)^{1/2} \\
 & \quad + \left(\sum_{f \in \E_h}
  \eta^{-1}\jumpterm^{-1}\|\{\nabla w\}\cdot\bnormal_f\|_{0,f}^2
  \right)^{1/2}\left(P_h^{new}(v,v)\right)^{1/2}.
\end{align*}
Gathering these inequalities together, we conclude that
\eqref{bounded1} and \eqref{bounded2} hold.  $\square$

\begin{theorem}\label{thm9}
For the exact solution $u \in V_{*h}$ of the model problem
\eqref{model-eq} and its SIP-DG solution $u_h \in V_h$ of
$\eqref{sip-dg-new}$, the following error estimate holds:
\begin{align}
   \|u - u_h\|_{\DG}^{new} \le C \inf_{y_h \in V_h}
    \|u - y_h\|_{\DG*}^{new},
    \label{error_est-new}
\end{align}
where the constant $C=C(\eta,C_4^{tr})$ is independent of $h$ and
the geometry of elements in $\T_h$.
\end{theorem}
\begin{proof}
By the consistency of $\asipnew$, we have
\begin{align*}
   \asipnew(u - u_h, \revise{w_h}) = 0, \qquad \forall
  \revise{w_h} \in V_h.
\end{align*}
\revise{
Thus, the discrete coercivity and boundedness yield
\begin{align*}
 \|u_h - y_h\|_{\DG}^{new} & \le 2 \sup_{w_h \in V_h}
    \frac{\asipnew(u_h - y_h,w_h)}{\|w_h\|_{\DG}^{new}}
   = 2 \sup_{w_h \in V_h}
  \frac{\asipnew(u - y_h,w_h)}{\|w_h\|_{\DG}^{new}} \\
 &  \le 2 C\|u - y_h\|_{\DG*}^{new}
\end{align*}
for an arbitrary $ y_h \in V_h$.  Therefore, we obtain
\begin{align*}
 \|u - u_h\|_{\DG}^{new} \le (1 + 2C) 
 \|u - y_h\|_{\DG*}^{new}, \qquad \forall y_h \in V_h.
\end{align*}
}
Taking the infimum for $y_h \in V_h$
\revise{and rewriting $C$}, we conclude that 
\eqref{error_est-new} holds.  $\square$
\end{proof}

\vspace{3mm}
To derive a more practical error estimation, we prepare another
general trace inequality.  Let $T$ be a $d$-simplex and $f$ be a
facet of $T$.  The following inequality holds
\cite[p.24]{Riviere}:
\begin{align}
  \|\nabla v\cdot\bnormal\|_{0,f} & \le 
  C_5^{tr} \frac{|f|^{1/2}}{|T|^{1/2}}
   \left(|v|_{1,T} + h_T|v|_{2,T}\right), \quad \forall v \in H^2(T).
   \label{trace-in3}
\end{align}

Now, let $\I_h^ku \in \PP_k(\T_h)$ be an interpolation of $u$ that
satisfies
\begin{align}
  \left[u - \I_h^k u\right]_f = 0 \qquad \forall f \in \E_h.
   \label{b-cond}
\end{align}
Note that the usual Lagrange interpolation satisfies 
\eqref{b-cond}.  Insert $\I_h^k u$ into $y_h$ in \eqref{error_est-new},
and set $U := u - \I_h^k u$.  Then, $[U]_f = 0$ on any $f \in \E_h$, and
$P_h^{new}(U,U) = 0$.  Therefore, we see that
\begin{align*}
  \|U\|_{\DG} = \left(\sum_{T \in \T_h} |U|_{1,T}^2\right)^{1/2}.
\end{align*}

For a facet $f \in \E_h^o$, let $f = T_f^{1} \cap T_f^{2}$,
 $T_f^{i} \in \T_h$, and
$U_i := U|_{T_f^{i}}$. The trace inequality \eqref{trace-in3} yields
\begin{align*}
  \jumpterm^{-1} \|\{\nabla U\}\cdot\bnormal_f\|_{0,f}^2
  & \le \frac{1}{2}\jumpterm^{-1} 
  \left(\|\nabla U_1\cdot\bnormal_f\|_{0,f}^2
  + \|\nabla U_2\cdot\bnormal_f\|_{0,f}^2 \right) \\
  & \le \frac{(C_5^{tr})^2}{2}\jumpterm^{-1} \Bigl(
   \frac{|f|}{|\tT_f^{1}|}(|U|_{1,\tT_f^{1}} + h_{\tT_f^1}
    |U|_{2,\tT_f^{2}})^2 \\
 & \hspace{3.5cm} + 
   \frac{|f|}{|\tT_f^2|}(|U|_{1,\tT_f^2} + h_{\tT_f^2}|U|_{2,\tT_f^2})^2
   \Bigr) \\
 & \le (C_5^{tr})^2 \left(
    |U|_{1,\tT_f^1}^2 + |U|_{1,\tT_f^2}^2 + h_{\tT_f^1}^2
      |U|_{2,\tT_f^1}^2
    + h_{\tT_f^2}^2 |U|_{2,\tT_f^2}^2 \right),
\end{align*}
and thus
\begin{align*}
 \sum_{f \in \E_h} \jumpterm^{-1}
     \|\{\nabla U\cdot\bnormal\}\|_{0,f}^2 
  \le (C_5^{tr})^2 \sum_{T \in \T_h} \left(|U|_{1,T}^2
   + h_T^2 |U|_{2,T}^2\right).
\end{align*}
That is, the following theorem has been proved.

\begin{theorem}\label{thm10}
Let $u \in V_{*h}$ be the exact solution of the model problem
\eqref{model-eq}, and $u_h \in V_h$ be its SIP-DG solution
 \eqref{sip-dg-new}.  Then, we have the following
error estimation:
\begin{align*}
 \|u - u_h\|_{\DG}^{new} \le C \|u - \I_h^k u\|_{\DG*}^{new} \le C 
   \sum_{T \in \T_h} \left(
   |u - \I_h^k u|_{1,T}^2 + h_T^2 |u - \I_h^k u|_{2,T}^2
   \right)^{1/2},
\end{align*}
where $\I_h^k u$ is an interpolation that satisfies \eqref{b-cond},
and the constant $C = C(C_4^{tr}, C_5^{tr}, \eta)$
is independent of $h$ and the geometry of elements in $\T_h$.
\end{theorem}

Let $\I_h^1 u$ be the usual Lagrange interpolation of $u$ and
let $R_T$ be the quantity defined in Section~\ref{sect:meshes}.
Suppose that $k = 1$, $d = 2$ and $u \in H^2(\Omega)$.   Then,
$|u - \I_h^1 u|_{2,T} = |u|_{2,T}$, and, from the results in
\cite{KobayashiTsuchiya1,KobayashiTsuchiya3},
\begin{align}
  |u - \I_h^1 u|_{1,T} \le C_{L1} R_T |u|_{2,T}, \qquad
   \forall u \in H^2(T),
  \label{L-est1}
\end{align}
where the constant $C_{L1}$ is independent of 
the geometry of $T$.  Thus, we find that
\begin{align*}
  \|u - u_h\|_{\DG}^{new} \le \|u - \I_h^k u\|_{\DG*}^{new} & \le C
  \sum_{T \in \T_h} (R_T^2 + h_T^2)^{1/2} |u|_{2,T} \\
  & \le C \max_{T \in \T_h} (R_T^2 + h_T^2)^{1/2} |u|_{2,\Omega}
  \le C (R^2 + h^2)^{1/2} |u|_{2,\Omega},
\end{align*}
where $R := \max_{T \in \T_h} R_T$ and $h := \max_{T \in \T_h} h_T$.

Suppose that $k \ge 2$, $d=2,3$, and $u \in H^{k+1}(\Omega)$.  Then,
from \cite{KobayashiTsuchiya3,KobayashiTsuchiya5,IshiKobaTsuchi1},
\begin{align}
  |u - \I_h^1 u|_{1,T} \le C_{Lk} R_T h_T^{k-1} |u|_{k+1,T}, \qquad
  |u - \I_h^1 u|_{2,T} \le C_{Lk} R_T^2 h_T^{k-3} |u|_{k+1,T},
  \label{L-est2}
\end{align}
and
\vspace{-1mm}
\begin{align*}
 \left(|u - \I_h^k u|_{1,\Omega}^2 + 
   \sum_{T \in \T_h} h_T^2 |u - \I_h^k u|_{2,T}^2
   \right)^{1/2} \le C  \max_{T \in \T_h}
  \left[R_T h_T^{k-2}\left(h_T^2 + R_T^2\right)^{1/2} \right]
  |u|_{k+1,\Omega},
\end{align*}
where the constant $C_{Lk}$ is independent of the geometry of $T$.
Therefore, we have obtained the following corollary.

\begin{corollary}\label{cor11}
 Let $u \in V_{*h}$ be the exact solution of the model problem
\eqref{model-eq}, and $u_h \in V_h$ be its SIP-DG solution
 \eqref{sip-dg-new}.  Suppose that $u \in H^{k+1}(\Omega)$.
Then, we have the following error estimations:
\begin{align*}
 \|u - u_h\|_{\DG}^{new} & \le C 
   \max_{T \in \T_h} (R_T^2 + h_T^2)^{1/2} |u|_{2,\Omega}
   \le C(R^2 + h^2)^{1/2}|u|_{2,\Omega} \quad (k=1,d=2), \\
  \|u - u_h\|_{\DG}^{new} & \le C 
   \max_{T \in \T_h} \left[R_T h_T^{k-2}(R_T^2 + h_T^2)^{1/2}
   \right] |u|_{k+1,\Omega} \\
  & \le CRh^{k-2}(R^2 + h^2)^{1/2}|u|_{k+1,\Omega} \quad
 (k\ge 2, d=2,3),
\end{align*}
where the constant $C = C(C_4^{tr}, C_5^{tr}, C_{Lk}, \eta)$ is
independent of $h$, $R$, and the geometry of elements in $\T_h$.
\end{corollary}

Let $d=2$.  In this case, elementary geometry (the law of sines) tells
us that a mesh $\T_h$ satisfies the maximum angle condition if and
only if there exists a constant $\sigma_2$ such that
\begin{align}
    \frac{R_T}{h_T} 
  \le \sigma_2, \qquad \forall T \in \T_h.
   \label{max-angle2}
\end{align}
\revise{
Recently, it is reported that the same situation holds for 
tetrahedrons.  That is, a tetrahedron satisfies the maximum angle
condition if and only if \eqref{max-angle2} holds 
\cite{IshiSuzKobaTsuc2}.  See also \cite{KobayashiTsuchiya6}.
}
Hence, we have the following corollary.

\begin{corollary}\label{cor12}
\revise{
Suppose that $k \ge 1$ if $d = 2$, and $k \ge 2$ if $d = 3$.
}
Let $u \in V_{*h}$ be the exact solution of the model
problem \eqref{model-eq}, and $u_h \in V_h$ be its SIP-DG solution
 \eqref{sip-dg-new}.  Suppose that $u \in H^{k+1}(\Omega)$ and
$\T_h$ satisfies the maximum angle condition.
Then, we have the following error estimation:
\begin{align*}
  \|u - u_h\|_{\DG}^{new}  \le C h^k |u|_{k+1,\Omega},
\end{align*}
where the constant $C$ depends only on $C_4^{tr}$, $C_5^{tr}$,
$C_{Lk}$, $\eta$, and $\sigma_2$.
\end{corollary}


\vspace{3mm}
For an $L^2$ error estimate, we have the following theorem, which is
quite similar to the Aubin--Nitsche lemma.
\begin{theorem}\label{thm13}
Let $d=2$.
Let $u \in V_{*h}$ be the exact solution of the model problem
\eqref{model-eq}, and $u_h \in V_h$ be its SIP-DG solution
 \eqref{sip-dg-new}.  Suppose that the model problem satisfies
the elliptic regularity \eqref{elli-reg}.  Then, we have the following
error estimation:
\begin{align*}
  \|u - u_h\|_{0,\Omega} \le 
 C R \|u - u_h\|_{\DG*}^{new},
\end{align*}
where $C = C(C_4^{tr}, C_5^{tr}, \eta, C_{ell}, C_{L1})$.
\end{theorem}
\noindent
\textit{Proof.} Setting $\phi := u - u_h$ in \eqref{bilinearform}, 
we consider an auxiliary problem: find $\zeta \in H_0^1(\Omega)$
such that
\begin{align*}
  a(\zeta, v) = \int_\Omega (u - u_h)v \dd x, \qquad
   \forall v \in H_0^1(\Omega).
\end{align*}
Because of the elliptic regularity \eqref{elli-reg}, we have
$\|\zeta\|_{2,\Omega} \le C_{ell} \|u-u_h\|_{0,\Omega}$.
Because $\zeta \in H^2(\Omega)$, we have
$[\nabla \zeta]\cdot\bnormal_f = \mathbf{0}$ on any $f \in \E_h^o$,
and $[\zeta] = 0$ on any $f \in \E_h$.   Hence, the definition
of $\asipnew$ implies that
\begin{align*}
  \asipnew(\zeta,u-u_h) & = \sum_{T \in T_h} 
  \int_T \nabla \zeta\cdot \nabla (u - u_h) \dd\bfx
  - \sum_{f \in \E_h} \int_f [u - u_h]\nabla \zeta\cdot
  \bnormal_f \dd  \\
 & = \int_{\Omega}(- \Delta \zeta)(u - u_h) \dd\bfx
  = \|u - u_h\|_{0,\Omega}^2.
\end{align*}
Let $\I_h^{1}\zeta$ be the piecewise linear Lagrange interpolation
of $\zeta$ on $\T_h$.
Because of the Galerkin orthogonality in Lemma~$3^{new}$,
we infer that
\begin{align*}
   \asipnew(u-u_h,\I_h^{1}\zeta) = 0.
\end{align*}
Therefore, it follows from \eqref{bounded2} and \eqref{L-est1} that
\begin{align*}
  \|u - u_h\|_{0,\Omega}^2 & = \asipnew(\zeta,u-u_h)
   = \asipnew(u-u_h, \zeta - \I_h^{1}\zeta) \\
  & \le \|u-u_h\|_{\DG*}^{new}
   \|\zeta - \I_h^{1}\zeta\|_{\DG*}^{new} \\
  & \le C_{L1} R \|u-u_h\|_{\DG*}^{new} |\zeta|_{2,\Omega} \\
  & \le C_{L1} C_{ell} R \|u-u_h\|_{\DG*}^{new}\|u - u_h\|_{0,\Omega}.
\end{align*}
This completes the proof. $\square$

\vspace{3mm}
\begin{corollary}\label{cor14}
Let $d = 2$. 
Let $u \in V_{*h}$ be the exact solution of the model problem
\eqref{model-eq}, and $u_h \in V_h$ be its SIP-DG solution
 \eqref{sip-dg-new}.  Suppose that elliptic regularity holds,
and $u \in H^{k+1}(\Omega)$.
Then, we have the following error estimation:
\begin{align*}
 \|u - u_h\|_{0,\Omega} & \le C 
    R (R^2 + h^2)^{1/2}|u|_{2,\Omega} \quad (k=1), \\
  \|u - u_h\|_{0,\Omega}  & \le C R^2h^{k-2}
  (R^2 + h^2)^{1/2}|u|_{k+1,\Omega} \quad  (k\ge 2),
\end{align*}
where the constant $C = C(C_4^{tr}, C_5^{tr}, \eta, C_{ell}, C_{L1})$ is
independent of $h$, $R$, and the geometry of elements in $\T_h$.
\end{corollary}

\vspace{3mm}
\noindent
\textit{Remark.}  It is conjectured that
Theorem~\ref{thm13} and Corollary~\ref{cor14} hold for $d = 3$.
To show this conjecture, we need to use a different interpolation, such
as the Crouzeix--Raviart interpolation, and analyze its error in
terms of the $\|\cdot\|_{DG*}^{new}$ norm.

\section{Numerical experiments for the new SIP-DG scheme}
\label{num-experi2}
In this section, we report the results of numerical experiments to
confirm the theoretical results obtained in the previous section.
First, we consider the same numerical experiment as in Section~3.3
(the same domain $\Omega$, same function $\phi$, and same meshes
$\T_h$) with the new SIP-DG scheme.  We fix $n=40$ and the penalty
parameter $\eta = 0.8$.
\footnote{
\revise{
If we set $\alpha=1$ in \eqref{adaptive}, we have
$\jumpterm = 12/h$.  Therefore, setting $\eta = 0.8$ 
in the proposed SIP-DG scheme approximately corresponds to setting 
$\eta = 10$ in the standard SIP-DG scheme in this case.
}
}
Then, the new SIP-DG method is
applied to the model problem \eqref{model-eq} with various $m$.  The
results are summarized in Table~\ref{tab2}.
\begin{table}[hbtp]
\caption{Errors produced using the new SIP-DG method.} \label{tab2}
\begin{tabular}[t]{|c|c|c|c|c|c|c|}
\hline
 $m$ & $h$ & $R$ & $L^2$-error & $H^1(\T_h)$-error & $P_h^{new}$-error & DG-error \\
\hline 
40 & 3.536e-2 & 1.768e-2 & 2.905e-4 & 3.630e-2 & 
\revise{2.037e-2} & \revise{4.162e-2} \\
\hline
80 & 2.795e-2 & 1.398e-2 & 1.719e-4 & 2.824e-2 &
\revise{1.635e-2} & \revise{3.133e-2} \\
\hline 
120 & 2.637e-2 & 1.318e-2 & 1.449e-4 & 2.630e-2 &
\revise{1.563e-2} & \revise{3.059e-2} \\
\hline 
160 & 2.578e-2 & 1.289e-2 & 1.347e-4 & 2.557e-2 &
\revise{1.531e-2} & \revise{2.962e-2} \\
\hline 
200 & 2.550e-2 & 1.275e-2 & 1.298e-4 & 2.522e-2 &
\revise{1.526e-2} & \revise{2.947e-2} \\
\hline 
400 & 2.512e-2 & 1.256e-2 & 1.234e-4 & 2.474e-2 &
\revise{1.509e-2} & \revise{2.898e-2} \\
\hline 
\end{tabular}
\end{table}
Note that the discretized solution $u_h$ is stable for all 
cases, and the errors decrease as $m$ increases, which is consistent
with the error estimations.

\revise{
To see how $\eta$ affects the performance of the proposed
SIP-DG scheme, we vary $\eta$ as $0.6, 0.5, 0.4, 0.3, 0.2, 0.1$
in the cases $m=40$ and $m=120$.  The proposed SIP-DG scheme is stable if
$\eta \ge 0.4$.  The scheme becomes unstable (more CPU time and
less accuracy) if $\eta \le 0.3$.  The ICCG method does not converge if
$\eta = 0.1$.
}

For the next experiment, we set $\Omega := (-1,1)\times (-1,1)$ and
$f(x,y) := \pi^2 \sin(\pi x) \sin(\pi y)$ in \eqref{model-eq}.
In this case, the exact solution is again
$u(x,y) = \sin(\pi x) \sin(\pi y)/2$.
We introduce the following Schwarz--Peano-type meshes for
$\Omega$. 

For a given positive integer $N$ and $\alpha > 1$, we
consider the isosceles triangle with base length $h:=2/N$ and height
$2/\lfloor 2/h^\alpha\rfloor \approx h^\alpha$ 
(then, $R_T \approx h^\alpha/2 + h^{2-\alpha}/8$), as depicted in
Figure~4.  We triangulate $\Omega$ with this triangle
as shown in Figure~5. 
Note that if $h \to 0$, these meshes satisfy neither the
shape-regularity condition nor the maximum angle condition.
Furthermore, if $\alpha \ge 2.0$ and $h \to 0$, the meshes do not
satisfy the circumradius condition.
See Kobayashi--Tsuchiya \cite{KobayashiTsuchiya2}.

\vspace{5mm}
\begin{tikzpicture}[line width = 1pt]
\newcommand{\yasuko}{2.0}
\newcommand{\ikuko}{0.7}
  \coordinate(A) at (-\yasuko,0);
  \coordinate(B) at (\yasuko,0);
  \coordinate(C) at (0,\ikuko);
  \draw[line width=1pt] (A) -> (B);
  \draw[line width=1pt] (B) -> (C);
  \draw[line width=1pt] (C) -> (A);
  \draw[line width=0.5pt,<->] (-\yasuko,-0.15) -> (\yasuko,-0.15)
   node[pos=0.5,below]{$h$};;
  \draw[line width=0.5pt,<->] (\yasuko+0.1,0.0) -> (\yasuko+0.1,\ikuko)
  node[pos=0.5,right]{$\approx h^\alpha$};;
\end{tikzpicture}
\quad \raise 8mm\hbox{\small{\textbf{Fig.~4} An isosceles triangle.}} \\

\includegraphics[width=5cm]{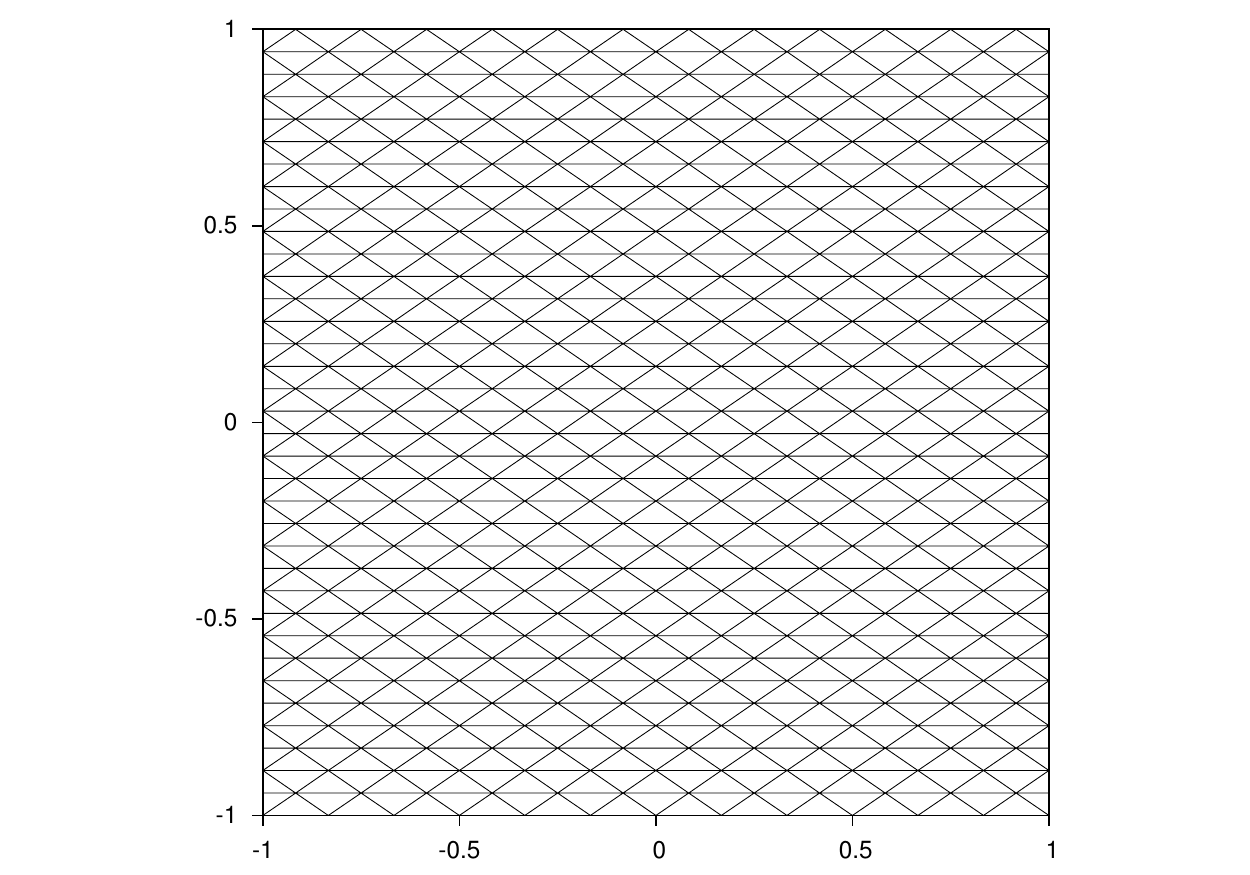} 
\hspace{-1mm}\raise 20mm
\hbox{
\begin{minipage}[c]{7cm}
\small{\textbf{Fig.~5} Schwarz--Peano-type mesh of $\Omega$. \\
$N=12$,  $\alpha=1.6$.}
\end{minipage}
}

\vspace{5mm}
We compute the SIP-DG solutions on this mesh with 
$\alpha = 1.2$, $1.5$, $1.8$, $2.0$, $2.1$ and various $N$,
and measure the errors. Figures~6 and 7 list the numerical results.
Tables~\ref{tab3} and \ref{tab4} contain concrete values of the numerical
results for $\alpha=1.5$ and $\alpha=2.0$, respectively.


\revise{
It seems that both errors behave similarly as $\alpha$ varies.
 Observing Figure~6 and Figure~7, we can infer that
$a_h^{(0)}(u-u_h,u-u_h)^{1/2}$ and $P_h^{(0)}(u-u_h,u-u_h)^{1/2}$
are $\mathcal{O}(R)$, that is, the errors
are governed by the parameter $R$, not $h$.  For example, when
$\alpha = 2.1$, the errors $a_h^{(0)}(u-u_h,u-u_h)^{1/2}$
and $P_h^{(0)}(u-u_h,u-u_h)^{1/2}$ increase as
$h$ decreases. However, the ratio $a_h^{(0)}(u-u_h,u-u_h)^{1/2}/R$
and $P_h^{(0)}(u-u_h,u-u_h)^{1/2}/R$ seem to be bounded by a constant.
}


\setcounter{figure}{5}
\begin{figure}[hbtp]
\hspace{-6mm}
\includegraphics[width=6.9cm]{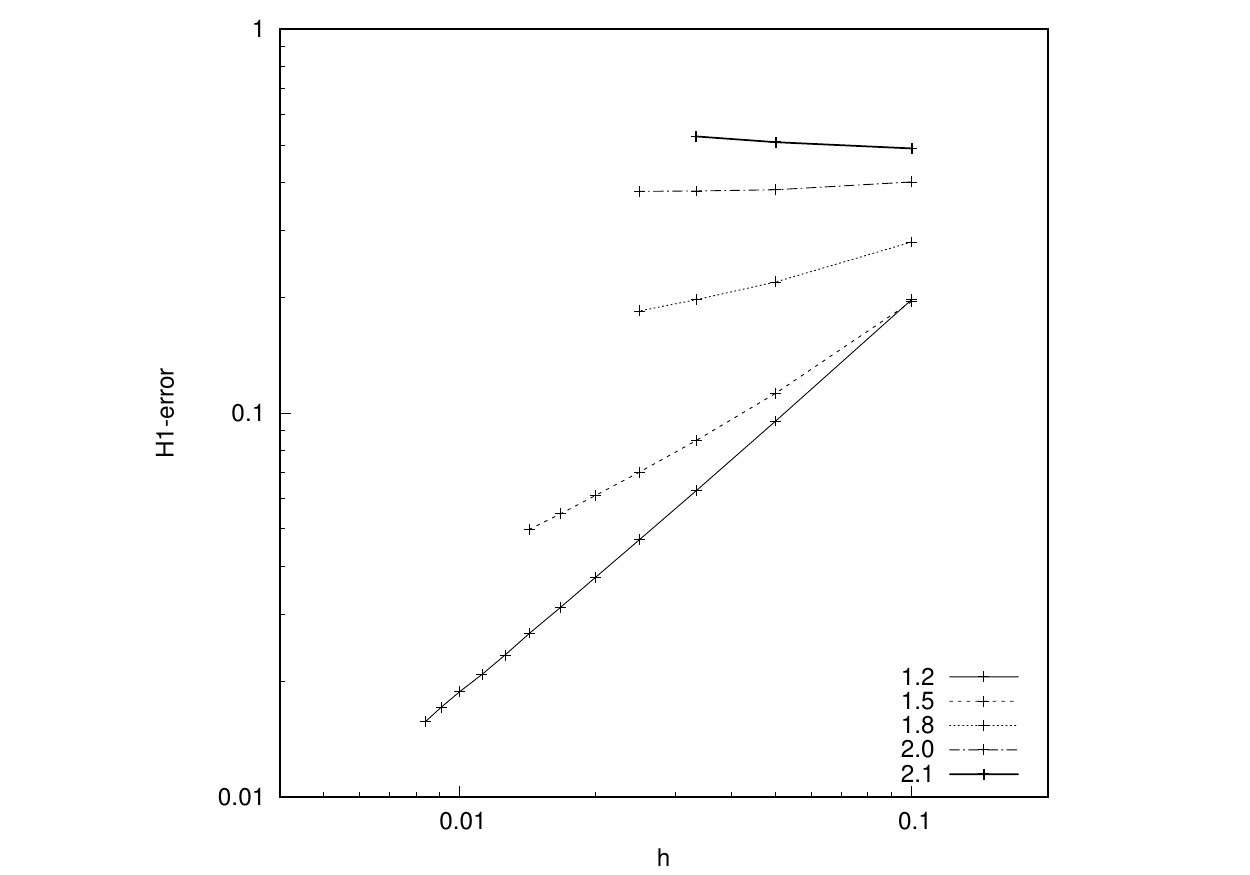} \hspace{-12mm}
\includegraphics[width=6.9cm]{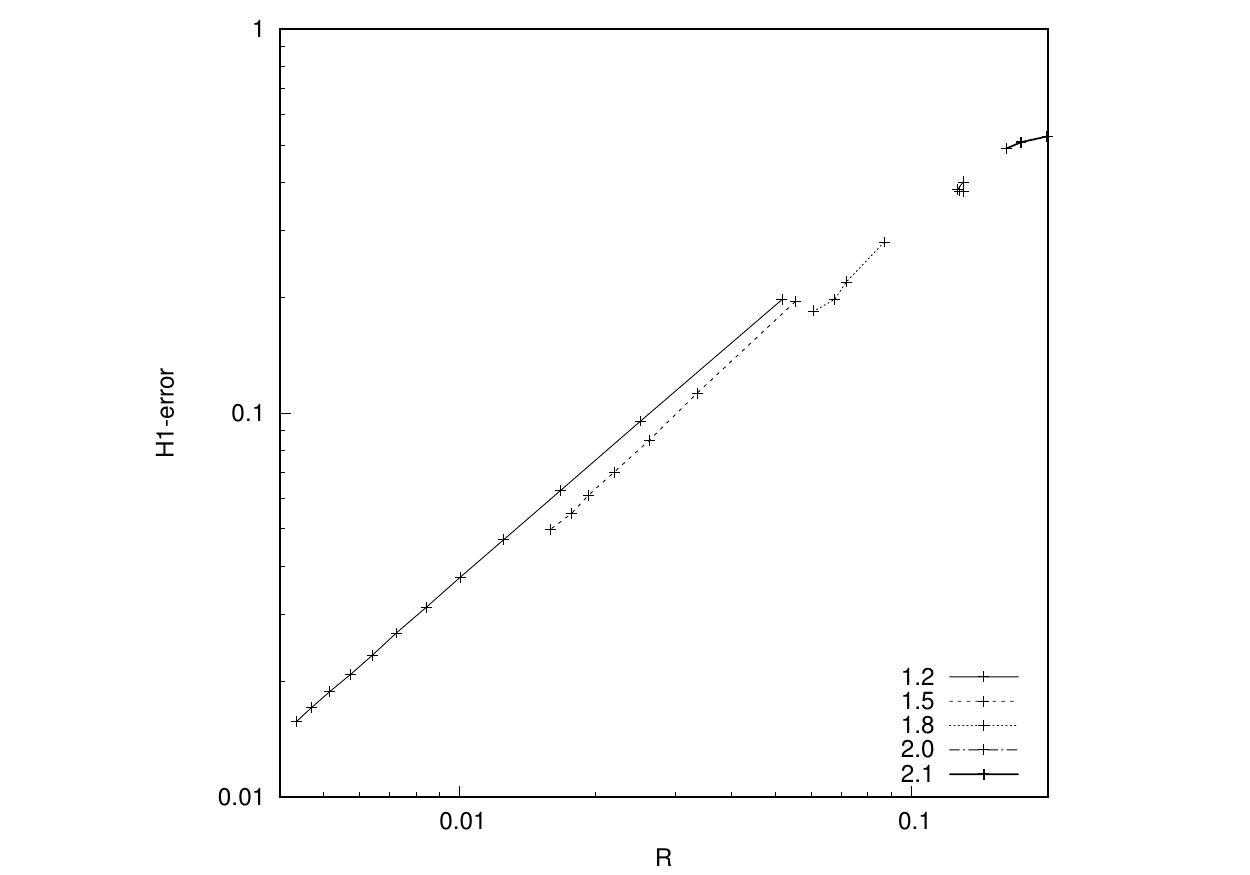} 
\caption{Behavior of the error $a_h^{(0)}(u-u_h,u-u_h)^{1/2}$  with
 respect to  $\alpha$.  The horizontal axis represents the maximum
 diameter (left) and the circumradius of the triangles (right), and the
 vertical  axis represents the error.  The legend
 indicates the value of $\alpha$ in each case.}
\end{figure}

\vspace{4mm}
\begin{figure}[hbtp]
\hspace{-6mm}
\includegraphics[width=6.9cm]{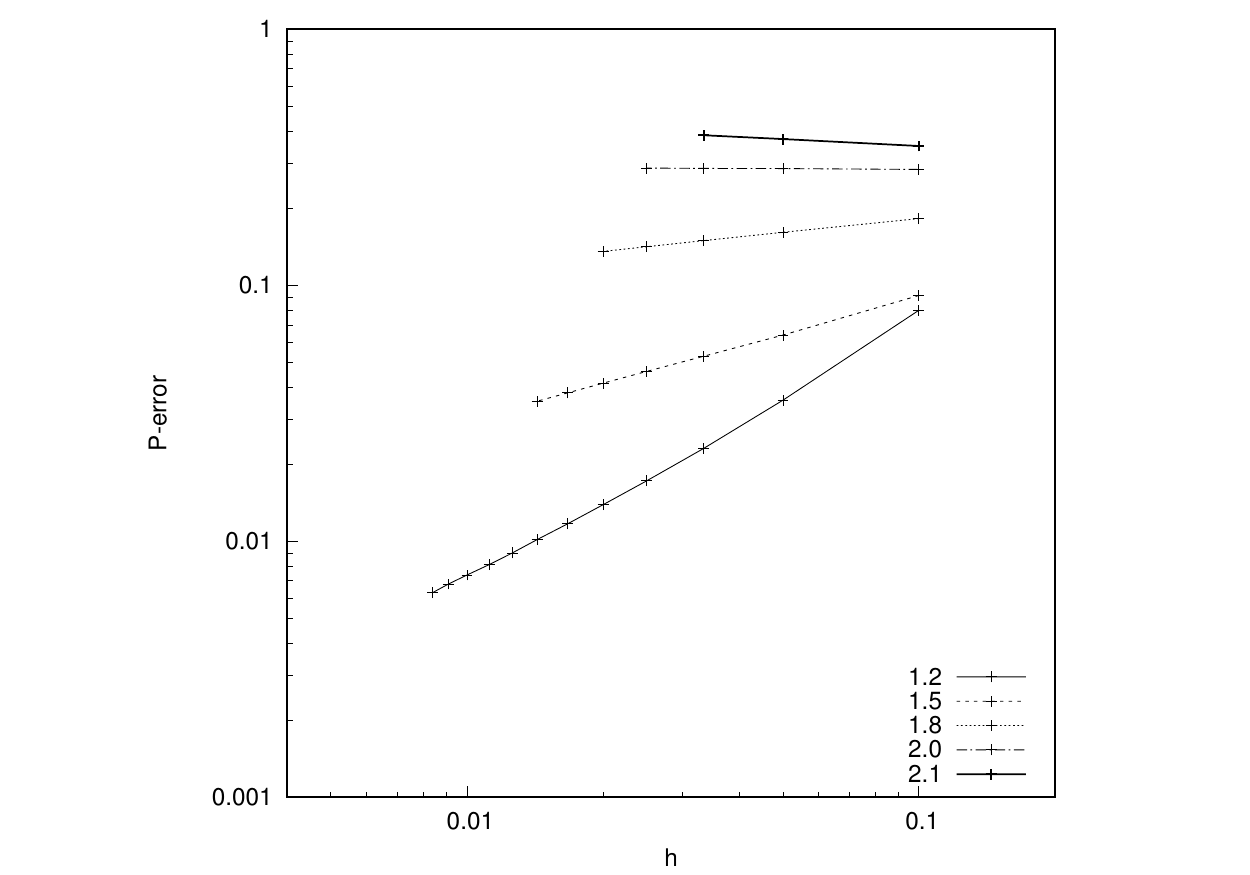} \hspace{-12mm}
\includegraphics[width=6.9cm]{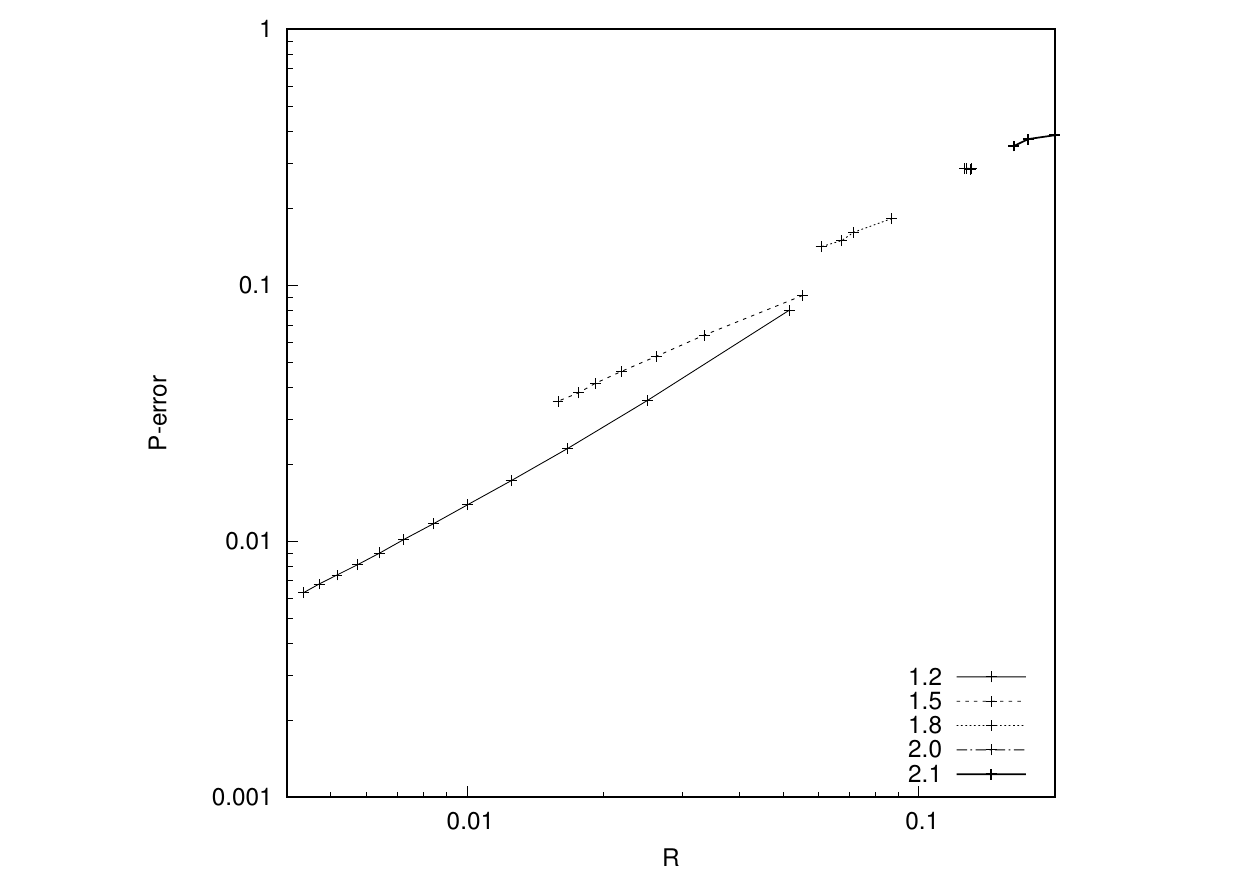} 
\caption{Behavior of the error
 $P_h^{new}(u-u_h,u-u_h)^{1/2}$  with respect to
 $\alpha$.  The horizontal axis represents the maximum diameter (left)
 and the circumradius of the triangles (right), and the vertical
 axis represents the error.  The legend  indicates the
 value of $\alpha$ in each case.}
\end{figure}

\begin{table}[hbtp]
\caption{Errors given by the new SIP-DG method on the mesh in
 Fig.~5 with  $\alpha = 1.5$.} \label{tab3}
\begin{tabular}[t]{|c|c|c|c|c|c|c|}
\hline
 $N$ & $h$ & $R$ & $L^2$-error & $H^1(\T_h)$-error & $P_h^{new}$-error & DG-error \\
\hline 
20 & 1.00e-1 & 5.528e-2 & 4.305e-3 & 1.954e-1 &
\revise{9.128e-2} & \revise{2.157e-1} \\
\hline
40 & 5.00e-2 & 3.350e-2 & 1.625e-3 & 1.128e-1 &
\revise{6.388e-2} & \revise{1.296e-1} \\
\hline 
60 & 3.34e-2 & 2.624e-2 & 9.828e-4 &  8.492e-2 &
\revise{5.284e-2} & \revise{1.000e-1} \\
\hline 
80 & 2.50e-2 & 2.198e-2 &  6.992e-4 & 7.028e-2 &
\revise{4.608e-2} & \revise{8.404e-2} \\
\hline 
100 & 2.00e-2 & 1.926e-2 & 5.432e-4 & 6.115e-2 &
\revise{4.149e-2} & \revise{7.390e-2} \\
\hline 
120 & 1.67e-2 & 1.765e-2 & 4.426e-4 & 5.471e-2 &
\revise{3.799e-2} & \revise{6.661e-2} \\
\hline 
140 & 1.43e-2 & 1.589e-2 & 3.735e-4 & 4.992e-2 &
\revise{3.527e-2} & \revise{6.112e-2} \\
\hline
\end{tabular}
\end{table}

\vspace{-10mm}
\begin{table}[hbtp]
\caption{Errors given by the new SIP-DG method on the mesh in Fig.~5
 with  $\alpha = 2.0$.} \label{tab4}
\begin{tabular}[t]{|c|c|c|c|c|c|c|}
\hline
 $N$ & $h$ & $R$ & $L^2$-error & $H^1(\T_h)$-error & $P_h^{new}$-error & DG-error \\
\hline 
20 & 1.00e-1 & 1.300e-1 & 2.395e-2 & 4.004e-1 &
\revise{2.836e-1} & \revise{4.906e-1} \\
\hline
40 & 5.00e-2 & 1.263e-1 & 2.296e-2 & 3.823e-1 &
\revise{2.857e-1} & \revise{4.772e-1} \\
\hline 
60 & 3.34e-2 & 1.273e-1 & 2.282e-2 & 3.792e-1 &
\revise{2.864e-1} & \revise{4.752e-1} \\
\hline 
80 & 2.50e-2 & 1.305e-1 & 2.283e-2 & 3.786e-1 &
\revise{2.870e-1} & \revise{ 4.751e-1} \\
\hline 
\end{tabular}
\end{table}

\section{Conclusion}
Using the general trace inequality \eqref{trace-in1}, we have presented
a new penalty term \eqref{newpenalty} and a new SIP-DG scheme
\eqref{newa}.  We have established its error estimates without imposing
the shape-regularity condition.  From numerical experiments, we have
confirmed that the new SIP-DG scheme is robust on anisotropic meshes.

The main idea in this paper is to use $\jumpterm$ instead of
the standard term 
$\frac{1}{h_f}$ in the penalty term.  Because this idea is very simple,
we expect that it can be extended in many directions and will be
used in many DG schemes.  In the following, we mention some directions
that are immediately apparent.
\begin{itemize}
 \item Application to many variants of DG methods.
 \item Application of the new SIP-DG scheme to non-proper meshes, and
the establishment of error estimations for these cases.
 \item Investigation of the possibility of applying the proposed
approach to hybridized DG methods.
\end{itemize}

\vspace{5mm}
\noindent
\textbf{Acknowledgements}
The authors were supported by JSPS KAKENHI Grant Numbers
17K14230, 20K14357, 16H03950, and 20H01820.
\revise{The authors would like to thank the anonymous referee for the
valuable comments. }

\end{document}